\documentclass[12pt]{article}
\usepackage{amsmath}
\usepackage{amsfonts}
\usepackage{amssymb}
\usepackage{amsthm}
\usepackage[T1]{fontenc}

\usepackage[left=1in,right=1in,bottom=1in,top=1in]{geometry}
\usepackage{amsmath}
\usepackage{amsfonts}
\usepackage{amssymb}
\usepackage{amsthm}
\usepackage{young}
\usepackage[vcentermath]{youngtab}
\usepackage{tikz}

\usepackage{graphicx}

\newcommand{\prob}[1]{\mathbb{P}\left( #1 \right)}

\newtheorem{theorem}{Theorem}

\newtheorem{remark}[theorem]{Remark}
\newtheorem{definition}[theorem]{Definition}

\def\unprotectedboldentry#1{\textcolor{Red}{\large{#1}}}
\def\boldentry{\protect\unprotectedboldentry}
\usepackage{tikz}
\usetikzlibrary{calc}
\usepackage{ifthen}
\newcommand{\tikztableauinternal}[1]{
    \def\newtableau{#1}
    \coordinate (x) at (-0.5,0.5);
    \coordinate (y) at (-0.5,0.5);
    \foreach \row in \newtableau {
        \coordinate (x) at ($(x)-(0,1)$);
        \coordinate (y) at (x);
        \foreach \entry in \row {
            \ifthenelse{\equal{\entry}{X}}
               {
                \node (y) at ($(y) + (1,0)$) {};
                \fill[color=gray!10] ($(y)-(0.5,0.5)$) rectangle +(1,1);
                \draw[color=gray, dotted] ($(y)-(0.5,0.5)$) rectangle +(1,1);
               }
               {
                \ifthenelse{\equal{\entry}{\boldentry X}}
                   {
                    \node (y) at ($(y) + (1,0)$) {};
                    \fill[color=gray] ($(y)-(0.5,0.5)$) rectangle +(1,1);
                    \draw ($(y)-(0.5,0.5)$) rectangle +(1,1);
                   }
                   {
                    \node (y) at ($(y) + (1,0)$) {\entry};
                    \draw ($(y)-(0.5,0.5)$) rectangle +(1,1);
                   }
               }
            }
        }
}

\newlength\cellsize 
\savebox2{%
\begin{picture}(12,12)
\put(0,0){\line(1,0){12}}
\put(0,0){\line(0,1){12}}
\put(12,0){\line(0,1){12}}
\put(0,12){\line(1,0){12}}
\end{picture}}

\newcommand\cellify[1]{\def\thearg{#1}\def\nothing{}%
\ifx\thearg\nothing
\vrule width0pt height\cellsize depth0pt\else
\hbox to 0pt{\usebox2\hss}\fi%
\vbox to 12\unitlength{
\vss
\hbox to 12\unitlength{\hss$#1$\hss}
\vss}}

\newcommand\tableau[1]{\vtop{\let\\=\cr
\setlength\baselineskip{-12000pt}
\setlength\lineskiplimit{12000pt}
\setlength\lineskip{0pt}
\halign{&\cellify{##}\cr#1\crcr}}}

\title{Continuity for Limit Profiles of Reversible Markov Chains}
\date{}

\author{Evita Nestoridi$^{\ast}$
}

\begin{document}

\maketitle

\begin{abstract}
We prove that the limit profile of a sequence of reversible Markov chains exhibiting total variation cutoff is a continuous function, under a computable condition involving the spectrum of the transition matrix and the cutoff window.

\end{abstract}


\let\thefootnote\relax
\footnotetext{  $^{\ast}$ \textit{Stony Brook University, United States. email}: evrydiki.nestoridi@stonybrook.edu, \\ 
\phantom{.} \hspace{0.85cm}Supported by the Simons Foundation Travel Support for Mathematicians award, MPSTSM00007955.}

\section{Introduction}
A central question in Markov chain mixing is the occurrence of cutoff, a phenomenon according to which a Markov chain converges abruptly to the stationary measure, see for instance \cite{BHP, DLP, PeSa, JS}. The focus of this paper is the limit profile of a Markov chain that exhibits cutoff, which captures the exact shape of the distance of the Markov chain from stationarity. Developments in studying limit profiles include introducing techniques for determining the limit profiles of Markov chains \cite{Teyssier, NT, lcomp, NOT} and determining the limit profile for famous Markov chains \cite{BaD,Teyssier, NT, Zhang, BuN, Lacoin, LP}. In this paper, we study the continuity properties of limit profiles.


 Let $X$ be a finite state space, and let $P$ be the transition matrix of an aperiodic and irreducible Markov chain on $X$. In other words, for $x,y \in X,$ the entry $P^t(x,y)$ is the probability of the walk starting at $x$ to be at $y$ after $t$ steps, where $t \in \{1, 2 , \ldots\}$.
We will consider the continuous time random walk associated with $P$, by considering the heat kernel $Q^t=e^{-t(I-P)}$. The probability measure $Q_x^t(\cdot)=Q^t(x, \cdot)$ converges to a unique measure $\pi(\cdot)$ on $X$ as $t$ goes to infinity.
We study this convergence with respect to total variation distance, which is defined as
$$d_x(t)=\Vert Q_x^t- \pi \Vert_{\textup{T.V.}} := \frac{1}{2} \sum_{y \in X} \vert Q_x^t(y)- \pi(y) \vert, $$
for $x \in X$.
\begin{definition}
Let $\varepsilon \in (0,1)$. The $\varepsilon$--mixing time with respect to the total variation distance, when starting at a state $x \in X,$ is defined as 
$$t^x_{\textup{mix}}(\varepsilon):= \inf \{ t :   d_x(t)   \leq \varepsilon \}.$$
\end{definition}

Cutoff describes a phase transition: as we run the family of Markov chains on $X_n$, the total variation distance is almost equal to $1$, and then suddenly it drops and approaches zero as $n$ goes to infinity. The formal definition of cutoff is given below.
\begin{definition}\label{cutoff}
A family of Markov chains on $X_n$ is said to have cutoff at time $t_n$ with window $w_n=o(t_n)$ if and only if
$$\lim_{c \rightarrow \infty} \lim_{n \rightarrow \infty} d_x^{(n)}(t_n-cw_n)= 1 \mbox{ and } \lim_{c \rightarrow \infty} \lim_{n \rightarrow \infty} d_x^{(n)}(t_n+cw_n)= 0,$$ 
where $d_x^{(n)}(t)$ denotes the total variation distance of the $n$--th Markov chain after $t$ steps starting at $x$.
\end{definition}

Given a Markov chain exhibiting cutoff, one can ask for more precise control on the exact distance from stationarity. This is known as the limit profile with respect to the sequences $(t_n)$ and $(w_n)$, defined as:
\begin{align} \label{limitprofile} \Phi_x(c):=\lim_{n \rightarrow \infty} d_x^{(n)}\left( t_{n} + c w_{n} \right), \mbox{ for all $c \in \mathbb{R}$, } 
\end{align}
if this limit exists.

The limit profile is known for only a few  Markov chains, such as the riffle shuffle \cite{BaD}, the asymmetric exclusion process on the segment \cite{BuN}, the simple exclusion process on the cycle \cite{Lacoin}, and the simple random walk on Ramanujan graphs \cite{LP}, etc. Teyssier \cite{Teyssier} determined the limit profile for random transpositions. Using representation theory of the symmetric group $S_n$, he used Fourier transform arguments for studying limit profiles that work for random walks on groups using a generating set that is a conjugacy class. The same limit profile is proved to be true for the $k$-cycle card shuffle \cite{NT}, under the assumption that $k=o(n)$ and for star transposition \cite{lcomp}. Olesker--Taylor and Schmid also studied the limit profile for the Bernoulli--Laplace chain \cite{SOT}. In Section \ref{exa} we give a more extensive presentation of the above results.

Furthermore, if $P$ is aperiodic, irreducible and reversible on $X$, then the spectrum of $P$ satisfies
\[ -1<\beta_{\vert X \vert} \leq \ldots \leq \beta_2< \beta_1=1,\]
and there is an orthonormal eigenbasis $\{f_j:X_n \rightarrow \mathbb{R}\}$. The main assumption on $P$ is that there is a continuous (or bounded) function $g$ on $\mathbb{R}$ such that
 \begin{align}
 \label{cond} \limsup_{n \rightarrow \infty}  w_n^2\sum_{i=2}^{\vert X_n \vert} f_i(x)^2 
   \left( 1-\beta_i\right)^2  e^{-2(t_n + c w_n)(1-\beta_i)} \leq g(c),
 \end{align}
for all $c \in \mathbb{R}$.
\begin{theorem}\label{main}
Let $P_n$ be the transition matrix of a reversible Markov chain on $X_n$ that exhibits cutoff at $t_n$ with window $w_n$ and satisfies \eqref{cond}. Assume that $\Phi_x$, the total variation limit profile with respect to the sequences $(t_n)$ and $(w_n)$ at $x$, exists. Then $\Phi_x$ is continuous for every $x \in X$.
\end{theorem}

For transitive chains, \eqref{cond} can be simplified. Assume that there is a continuous (or bounded) function $g$ on $\mathbb{R}$ such that
 \begin{align}
 \label{cond3} \limsup_{n \rightarrow \infty}  w_n^2\sum_{i=2}^{\vert X_n \vert} 
   \left( 1-\beta_i\right)^2  e^{-2(t_n + c w_n)(1-\beta_i)} \leq g(c),
 \end{align}

\begin{theorem}\label{main3}
Let $P_n$ be the transition matrix of a transitive, reversible Markov chain on $X_n$ that exhibits cutoff at $t_n$ with window $w_n$ and satisfies \eqref{cond3}. Assume that $\Phi_x$, the total variation limit profile with respect to the sequences $(t_n)$ and $(w_n)$ at $x$, exists. Then $\Phi_x$ is continuous for every $x \in X$.
\end{theorem}

Teyssier has proven that any decreasing function can be attained as a limit profile (private communication), and thus an assumption such as \eqref{cond} or \eqref{cond3} is needed.
The main idea to prove Theorem \ref{main} is to bound the difference $\vert \Phi_x(c)- \Phi_x(d) \vert$ by an $\ell_2$ expression involving the spectrum of $P$, simplify this expression using the Mean Value Theorem and use the assumptions to prove that this upper bound goes to zero as $c $ goes to $d$.

In Section \ref{prem}, we give the necessary setup. The proof of Theorem \ref{main} can be found in Section \ref{next}. In Section \ref{disc} we consider the discrete time limit profile. Passing through the spectral decomposition, we explain how to prove that the discrete limit profile is also continuous (see Theorem \ref{main2}). Section \ref{tran} discusses the proof of Theorem \ref{main3} for transitive chains, both for continuous and discrete time chains. Section \ref{exa} discusses a variety of examples, some of them satisfying the assumptions of Theorems \ref{main} and \ref{main2} (and therefore have continuous limit profiles) and some that don't, but still have continuous limit profiles.

\section{Preliminaries}\label{prem}
Let $P$ be an irreducible Markov chain on $X_n$, that is reversible with respect to the stationary measure $\pi$. $P$ has real eigenvalues satisfying
\[ -1\leq \beta_{\vert X_n \vert}\leq \ldots \leq\beta_2<\beta_1=1,\]
and there is an orthonormal eigenbasis $\{ f_i : X_n \rightarrow \mathbb{C}\}$, which we can use to write
\[Q^t_x(y) = \pi(y)\sum_{i=1}^{\vert X_n \vert} f_i(x) f_i(y) e^{-t(1-\beta_i)},\]
as explained in Lemma 12.2 of \cite{LPW}. 

Let $c_1 < c_2 $ be two real numbers. We set $s_1=t_n +c_1 w_n$ and $s_2=t_n +c_2 w_n$ and the orthonormality of the $f_i$'s gives
\begin{align}
  \label{l2}  \bigg \Vert\frac{Q^{s_1}_x- Q^{s_2}_x}{\pi} \bigg \Vert_{2,\pi}^2= \sum_{i=1}^{\vert X_n \vert} f_i(x)^2 \left( e^{-s_1(1-\beta_i)} - e^{-s_2(1-\beta_i)}\right)^2,
\end{align}
where $\Vert \cdot \Vert_2$ is the $\ell_2$ norm with respect to the stationary measure $\pi$.
Using the triangle inequality, we get 
\begin{align*}
     \vert d_x(s_1)- d_x(s_2)\vert \leq \Vert Q^{s_1}_x- Q^{s_2}_x\Vert_{T.V}.
    \end{align*}
An application of the Cauchy-Schwarz inequality, thus gives
\begin{align*}
    4 \vert d_x(s_1)- d_x(s_2)\vert^2 \leq \bigg \Vert\frac{Q^{s_1}_x- Q^{s_2}_x}{\pi} \bigg \Vert_{2,\pi}^2.
    \end{align*}
    We use \eqref{l2} to rewrite this as
\begin{align*}
    4 \vert d_x(s_1)- d_x(s_2)\vert^2 \leq \sum_{i=1}^{\vert X_n \vert} f_i(x)^2 \left( e^{-s_1(1-\beta_i)} - e^{-s_2(1-\beta_i)}\right)^2.
\end{align*}
Since $\beta_1 =1$, the above sum really starts from $i=2$. The Mean Value Theorem for the functions $h_i(x)=e^{-x(1-\beta_i)} $ gives that there are $d_i=d_i(n) \in ( c_1, c_2)$ such that
    \begin{align*}
   4\vert d_x(s_1)- d_x(s_2)\vert^2 \leq (c_2-c_1)^2 w_n^2\sum_{i=2}^{\vert X_n \vert} f_i(x)^2 
   \left( 1-\beta_i\right)^2  e^{-2s_i(1-\beta_i)},
    \end{align*}
    where $s_i= t_n +d_iw_n$. Using the fact that $s_i \geq s_1$ we have that
    \begin{align}
    \label{bb}4\vert d_x(s_1)- d_x(s_2)\vert^2 \leq (c_2-c_1)^2 w_n^2\sum_{i=2}^{\vert X_n \vert} f_i(x)^2 
   \left( 1-\beta_i\right)^2  e^{-2s_1(1-\beta_i)}.
    \end{align}

\section{The proof of Theorem \ref{main}}\label{next}
\begin{proof}[Proof of Theorem \ref{main}]
Let $P$ be the transition matrix of a Markov chain on $X_n$ satisfying the conditions of Theorem \ref{main}.
Similarly as in Section \ref{prem}, let $c_1< c_2 $ be two real numbers. We set $s_1=t_n +c_1 w_n$ and $s_2=t_n +c_2 w_n$.  

Letting $n \rightarrow \infty$ in \eqref{bb}, and using the assumption of Theorem \ref{main},
 \begin{align*}
     4\vert \Phi_x(c_1) -  \Phi_x(c_2)\vert^2
    & \leq  g(c_1) \left(c_1-c_2\right)^2.
    \end{align*}
Letting $c_1 \rightarrow c_2$ or ($c_2 \rightarrow c_1$) on both sides will also push $c_3 \rightarrow c_2$ or $c_1$ respectively. The continuity of $g$ shows that $\Phi_x$ is continuous.




\end{proof}
\section{The discrete time limit profile}\label{disc}
We can also consider mixing times and the limit profiles for discrete time Markov chains and discuss continuity as follows. Recall that $P$ is the transition matrix of an aperiodic and irreducible Markov chain on $X$, which is reversible with respect to the stationary measure $\pi$.
The total variation distance from stationarity is defined as
$$d_x(t)=\Vert P_x^t- \pi \Vert_{\textup{T.V.}} := \frac{1}{2} \sum_{y \in X} \vert P_x^t(y)- \pi(y) \vert $$
for $x \in X$.
The mixing time with respect to the total variation distance, when starting at $x \in X$ is defined as 
$$t^x_{\textup{mix}}(\varepsilon):= \min \{ t :   d_x(t)   \leq \varepsilon \},$$
for every $\varepsilon \in (0,1)$.
Cutoff and the limit profile are defined just as in Definition \ref{cutoff} and \eqref{limitprofile}. We will now give an alternative definition of limit profiles, so that we can allow $c$ to be a real number as supposed to an integer.

Since $P$ is an irreducible and aperiodic Markov chain in $X_n$ that is reversible with respect to the stationary measure $\pi$, $P$ has real eigenvalues that satisfy
\[ -1\leq \beta_n\leq \ldots \leq\beta_2<\beta_1=1\]
and there is an orthonormal eigenbasis $f_i$, which we can use to write
\[P^t_x(y) = \pi(y)\sum_{i=1}^{\vert X_n \vert} f_i(x) f_i(y) \beta_i^t\]
as explained in Lemma 12.2 of \cite{LPW}. Therefore,
\[d_x(t) = \frac{1}{2}\sum_{y \in X_n} \pi(y) \bigg \vert \sum_{i=2}^{\vert X_n \vert} f_i(x) f_i(y) \beta_i^t \bigg \vert .\]
If $P$ exhibits cutoff at $t_n$ with window $w_n$, we can define the discrete time limit profile with respect to the sequences $(t_n)$ and $(w_n)$ as
\[\Phi(c):= \lim_{n \rightarrow \infty}\frac{1}{2}\sum_{y \in X_n} \pi(y) \bigg \vert \sum_{i=2}^{\vert X_n \vert} f_i(x) f_i(y) \beta_i^{t_n +c w_n} \bigg \vert ,\]
if it exists. Under the assumption that all eigenvalues of $P$ are nonnegative, the limit profile can be viewed as a function whose domain is the real numbers and the following theorem makes sense for discrete Markov chains.

\begin{theorem}\label{main2}
Let $P_n$ be the transition matrix of a reversible Markov chain on $X_n$ that exhibits cutoff at $t_n$ with window $w_n$, whose eigenvalues are nonnegative. Assume that there is a continuous function $g$ (or just bounded) on $\mathbb{R}$ such that
\begin{align}
\label{cond2}
\limsup_{n\rightarrow \infty}w_n \bigg \lbrace \sum_{i=2}^{\vert X_n \vert} f^2_i(x)  \log^2 ( \beta_i )  \beta_i^{2(t_n + c w_n)} \bigg \rbrace^{1/2}  \leq g(c),
\end{align}
for every $c \in \mathbb{R}$. Assume that $\Phi_x$, the discrete time limit profile with respect to the sequences $(t_n)$ and $(w_n)$ at $x$, exists. Then $\Phi_x$ is continuous for every $x \in X$.
\end{theorem}
\begin{proof}
Let $c_1, c_2 \in \mathbb{R}$. We set $s_1=t_n +c_1 w_n$ and $s_2=t_n +c_2 w_n$ and we write
\begin{align}
 \label{one} \bigg \vert \sum_{y \in X_n} \pi(y) \bigg \vert \sum_{i=2}^{\vert X_n \vert} f_i(x) f_i(y) \beta_i^{s_1} \bigg \vert - \sum_{y \in X_n} \pi(y) \bigg \vert \sum_{i=2}^{\vert X_n \vert} f_i(x) f_i(y) \beta_i^{s_2} \bigg \vert \bigg \vert.
\end{align}
Using the triangle inequality, we have 
\begin{align*}
    \eqref{one} \leq  \sum_{y \in X_n} \pi(y) \bigg \vert \sum_{i=2}^{\vert X_n \vert} f_i(x) f_i(y) \left( \beta_i^{s_1}- \beta_i^{s_2} \right) \bigg \vert .
    \end{align*}
Just as in the continuous case, Cauchy--Schwarz and the orthonormality of the $f_j$'s gives 
\begin{align*}
    \eqref{one} & \leq  \bigg \lbrace \sum_{i=2}^{\vert X_n \vert} f^2_i(x)  \left( \beta_i^{s_1}- \beta_i^{s_2} \right)^2 \bigg \rbrace^{1/2}.
    \end{align*}
   The rest of the argument is very similar to the continuous case. The Mean Value Theorem for the functions $H_i(x)= \beta_i^x$ gives that there exist $d_i=d_i(n) \in (c_1,c_2)$ such that
    \begin{align*}
\eqref{one} & \leq  \vert c_2 - c_1 \vert w_n \bigg \lbrace \sum_{i=2}^{\vert X_n \vert} f^2_i(x)  \log^2 ( \beta_i )  \beta_i^{2s_i} \bigg \rbrace^{1/2}  \\
& \leq  \vert c_2 - c_1 \vert w_n \bigg \lbrace \sum_{i=2}^{\vert X_n \vert} f^2_i(x)  \log^2 ( \beta_i )  \beta_i^{2s_1} \bigg \rbrace^{1/2} ,
    \end{align*}
    where $s_3= t_n +d_i w_n$.
Letting $n \rightarrow \infty $, and \eqref{cond2} give that
\begin{align*}
    \vert \Phi_x(c_1)- \Phi_x(c_2) \vert \leq \frac{1}{2} \vert c_2 - c_1 \vert g(c_1),
\end{align*}
for a continuous function $g$. Letting $c_1 \rightarrow c_2$ or $c_2 \rightarrow c_1$ we get that $\Phi$ is a continuous function.
\end{proof}

\section{Transitive chains}\label{tran}
For transitive chains, \eqref{cond} and \eqref{cond2} respectively take a nicer form. If we average over $x \in X$ both sides of \eqref{bb}, the orthonormality of the $f_j$'s and transitivity will give
\begin{align}
    \label{bbb}4\vert d_x(s_1)- d_x(s_2)\vert^2 \leq (c_2-c_1)^2 w_n^2\sum_{i=2}^{\vert X_n \vert} 
   \left( 1-\beta_i\right)^2  e^{-2s_3(1-\beta_i)},
    \end{align}
    for every $x \in X$. The rest of the argument follows exactly as before. This is why
for the continuous time setup, the assumption on $P$ is that there is a continuous (or bounded) function $g$ on $\mathbb{R}$ such that
 \begin{align*}
 \limsup_{n \rightarrow \infty}  w_n^2\sum_{i=2}^{\vert X_n \vert} 
   \left( 1-\beta_i\right)^2  e^{-2(t_n + c w_n)(1-\beta_i)} \leq g(c),
 \end{align*}
for all $c \in \mathbb{R}$.

 Similarly, for transitive, discrete time Markov chains, we have the following.
 \begin{theorem}\label{main4}
Let $P_n$ be the transition matrix of a transitive, reversible Markov chain on $X_n$ that exhibits cutoff at $t_n$ with window $w_n$, and whose eigenvalues are nonnegative. Assume that there is a continuous function (or bounded) $g$ on $\mathbb{R}$ such that
\begin{align}
\label{cond4}
\limsup_{n\rightarrow \infty}w_n \bigg \lbrace \sum_{i=2}^{\vert X_n \vert}  \log^2 ( \beta_i )  \beta_i^{2(t_n + c w_n)} \bigg \rbrace^{1/2}  \leq g(c),
\end{align}
for every $c \in \mathbb{R}$. Assume that $\Phi_x$, the discrete time limit profile with respect to the sequences $(t_n)$ and $(w_n)$ at $x$, exists.
Then $\Phi_x$ is continuous for every $x \in X$.
\end{theorem}
\section{Examples}\label{exa}
In this section, we give examples of Markov chains some of which satisfy the assumptions of Theorem \ref{main2} and their limit profile is known (and is indeed continuous). We also give examples, where the limit profile is not known to exist, but the assumptions of Theorem \ref{main2} are satisfied. In this case, if the limit profile exists, then it must be continuous. Finally, we provide examples of Markov chains whose limit profiles are known (and are continuous), but we don't know if \eqref{cond} or \eqref{cond2} is satisfied. 

\begin{remark}
For the following Markov chains, if the eigenvalues are not non-negative, then we can consider the lazy versions, that is with probability $1/2$ we do nothing and with probability $1/2$ we perform the transition matrix $P$.
\end{remark}

\subsection{Lazy random walk on the hypercube}
Consider the discrete hypercube $X_n=\{0,1\}^n$. The simple random walk suggests that we pick a coordinate uniformly at random and we update it to $0$ with probability $1/2/$, or we update it to $1$ otherwise. 

Example 2 of Chapter 3C of \cite{PersiBook} explains that the cutoff time is $\frac{1}{2} n \log n$ and the window is $n$ and that the eigenvalues are $\beta_i= 1- \frac{i}{n}$ for $1\leq i \leq n$. One can find a proof of the fact that
\[\Phi_x(c)= 2 \Phi \left(\frac{1}{2}e^{-c}\right)-1, \]
where $\Phi$ stands for the standard, normal distribution, in Appendix 5.1 of \cite{NT}. This is a continuous function.

Checking condition \eqref{cond4} is easy:
\begin{align}
\lim_{n\rightarrow \infty}n \bigg \lbrace \sum_{i=1}^{n} {n \choose i} \log^2 \left( 1-\frac{i}{n} \right)  \left(1- \frac{i}{n}\right)^{ n \log n+2 c n)} \bigg \rbrace^{1/2}  &\leq \lim_{n\rightarrow \infty}2n \bigg \lbrace \sum_{i=1}^{n} \frac{n^i}{i!} \frac{i^2}{n^2} e^{- i \log n - 2ci} \bigg \rbrace^{1/2}\cr
&\label{cont}\leq  \lbrace 2e^{e^{-c}}(e^{-2c} +e^{-c}) \rbrace^{1/2} ,
\end{align}
where in \eqref{cont} we use the fact that $xe^x +x^2e^x= \sum_{i=0}^{\infty} \frac{i^2 x^i}{i!}$. 
In particular, \eqref{cond4} is satisfied!

\subsection{Random Walks on Ramanujan Graphs}
Lubetzky and Peres \cite{LP} proved that the simple random walk on any Ramanujan graph on $n$ vertices and fixed degree $d \geq 3$ exhibits cutoff at $\frac{d}{d-2} \log_{d-1} n$ and window $\sqrt{\log_{d-1}n} $. They also determine that the limit profile is given by 
\[\Phi(c)= \prob{Z>\alpha c},\]
where $\alpha= \frac{(d-2)^{3/2}}{2 \sqrt{d(d-1)}}$ and $Z$ is a standard, normal random variable. In this example, \eqref{cond} or \eqref{cond2} are not satisfied. However, the limit profile is indeed a continuous function. 

On the other hand, for a version of a non--backtracking random walk on Ramanujan graphs with girth at least $\delta \log_{d-1} n$ (for fixed $\delta$), \cite{NeSa} proved that cutoff occurs at $\log_{d-1} n$ with constant window. The transition matrix of this walk is symmetric and the analysis is done via analyzing the $\ell_2$ norm with $w_n$ and all $\beta_i$'s are known to be constant on $n$, so that \eqref{cond2} is indeed satisfied. In particular, using equation $(25)$ of \cite{NeSa}, \eqref{cond2} turns out to be 
\[\lim_{n\rightarrow \infty}w_n \bigg \lbrace \sum_{i=2}^{\vert X_n \vert} f^2_i(x)  \log^2 ( \beta_i )  \beta_i^{2(t_n + c w_n)} \bigg \rbrace^{1/2}  \leq C_{d,\delta} \left( \frac{n}{(d-1)^{t_n}}\right)^{1/2},\]
where $C_{d,\delta}$ is a constant depending on $d$ and $\delta$.
Therefore, if the limit profile exists, Theorem \ref{main2} says that it must be continuous.

\subsection{Random--to--random}
Consider a deck of $n$ cards. According to the random--to--random card shuffle, one picks a card uniformly at random, removes it from the deck and inserts it back to a uniformly at random chosen position. Cutoff at $\frac{3}{4} \log n - \frac{1}{4} n \log \log n$ with window $n$ was proven in \cite{BN, Subag}. The transition matrix was diagonalized by Dieker and Saliola \cite{DSaliola} and this is how we know that $t_{\textup{rel}}=n$. All eigenvalues are non-negative as explained in \cite{BN}. The upper bound analysis of \cite{BN} is in terms of the $\ell_2$ norm, and can be adjusted to prove that condition \eqref{cond4} is met, because adjusting equation $(25)$ of \cite{BN} gives
\[\lim_{n\rightarrow \infty}w_n^2 \sum_{i=2}^{\vert X_n \vert}  \log^2 ( \beta_i )  \beta_i^{2(t_n + c w_n)}  \leq 3e^{-2c}+ 2 \sum_{\ell=0}^{\infty} \ell^2 e^{-2(c-1) \ell}= 3e^{-2c} + \frac{4 e^{-2(c-1)}}{(1-e^{-2(c-1)})^3},\]
for $c>1$.
The limit profile of random--to--random (if it exists) is not known. Theorem \ref{main4} says that if the limit profile of random--to--random exists, it must be a continuous function for $c > 1$.

\subsection{Random and star transpositions, and $k$--cycles}
Consider a deck of $n$ cards. Pick two cards uniformly at randomly at random and independently and swap them. This is the random transpositions shuffle. Diaconis and Shahshahani proved that it exhibits cutoff at $t_n= \frac{1}{2} n \log n$ with window $w_n= \frac{n}{2}$, which is also equal to $t_{\textup{rel}}$. Their proof relys on studying the $\ell_2$ distance.

Teyssier \cite{Teyssier} determined the limit profile for random transpositions is given by
\begin{align}
\label{lim}
    \Phi(c)= d_{\textup{T.V.}} \left( \textup{Poisson(1)},\textup{Poisson}(1+e^{-c})\right),
\end{align}
where $d_{\textup{T.V.}}$ is the total variation distance. This is a continuous function. To check if condition \eqref{cond4} holds, we would need to consider the lazy version of the walks, which would shift the cutoff time. The $\ell_2$ analysis, however, is very similar and in particular \eqref{cond4} can be shown to be satisfied.

The star transposition shuffle suggests to pick a card uniformly at random and transpose it with the top card. Flatto, Odlyzko and Wales \cite{FOW} proved that $t_{\textup{rel}}= n$ and Diaconis \cite{SFlour} proved that the star transposition shuffle exhibits cutoff at $n \log n$ with window $n$. The proof is also spectral, meaning that \eqref{cond4} is satisfied, after we consider the lazy version of the card shuffle. In \cite{lcomp}, it was proven that the limit profile for star transpositions is also given by \eqref{lim}, which is a continuous function.

Hough \cite{Hough} studied the $k$--cycle shuffle which is the random walk on $S_n$ generated by the conjugacy class of $k$--cycles. He proved that for $k=o(n/ \log n)$, the $k$--cycle shuffle exhibits cutoff at $\frac{1}{k} n \log n $ with window $w_n=t_{\textup{rel}}= \frac{n}{k}$. In particular, his analysis can be adjusted to prove that \eqref{cond4} is satisfied, after we consider the lazy version of the shuffle. In \cite{NT}, the limit profile is also proven to be given by \eqref{lim}.

\subsection{The Bernoulli--Laplace chain}
Let $n \in \mathbb{N} $ and $ 1 \ll k \leq n/2$.
On the Bernoulli--Laplace model, we consider two urns, one containing initially $k$ red balls and the other one containing $n-k$ blue balls. At each step, we pick a ball from each urn uniformly at random and we swap them. Diaconis and Shahshahani \cite{Ber1} proved that this chain exhibits cutoff at $\frac{1}{2} n \log \min \{k, \sqrt{n} \}$ with window $n$. For $ k$ of order $n$, the eigenvalues are positive and adjusting the $\ell_2$ analysis of \cite{Ber1}, we get that 
\[\lim_{n\rightarrow \infty}w_n^2 \sum_{i=2}^{\vert X_n \vert}  \log^2 ( \beta_i )  \beta_i^{2(t_n + c w_n)}  \leq A \lim_{n\rightarrow \infty} n^2 \sum_{i=1}^{k} \frac{i^2(n-i+1)^2}{k^2(n-k)^2}  \frac{e^{-ci+ \frac{i(i-1)}{n}}}{i!},\]
for some constant $A$.
Using the fact that $k$ is of order $n$, we get
\[\lim_{n\rightarrow \infty}w_n^2 \sum_{i=2}^{\vert X_n \vert}  \log^2 ( \beta_i )  \beta_i^{2(t_n + c w_n)}  \leq A'  \sum_{i=1}^{\infty} i^2 \frac{e^{-(c-1)i}}{i!},\]
for another constant $A'$. The right hand side is continuous for every $c \in \mathbb{R}$. This implies that the limit profile is continuous.

Indeed, Olesker--Taylor and Schmid \cite{SOT} proved that in this regime, the limit profile is 
\[ \Phi (c) = \Vert N(e^{-2c},1)-N(0,1)\Vert_{\textup{T.V.}},\]
where $N(\mu, 1)$ is a normal distribution with mean $\mu$ and variance one. And this is a continuous function.

\subsection{Biased Card Shuffling and ASEP}
The Metropolis biased card shuffling (also called the multi-species ASEP on a finite interval or the random Metropolis scan) is a card shuffle where if the cards at positions $i$ and $i+1$ are sorted, then with rate 1 we swap them; otherwise, we swap them with rate $q$. This Markov chain is known to be reversible with respect to the Mallows measure of size $n$. Labb\'e and Lacoin \cite{LL} proved that it exhibits cutoff at $\frac{2}{1-q}n$ and Zhang \cite{Zhang} proved that the window is $2^{-2/3}n^{1/3}$ and that the limit profile is given by
\[\Phi(c)= 1- F_{\textup{GOE}}(c),\]
where $F_{\textup{GOE}}$ is the Tracy-Widom GOE distribution.
The $\ell_2$ analysis is still open and condition \eqref{cond} is not known to be satisfied, yet knowing the limit profile formula, one can check that it is indeed continuous.

Similarly for the asymmetric exclusion process (ASEP), we consider $k$ particles on a segment of size $n$ Each particle waits an exponential time with
parameter 1, after which with fixed probability $p > 1/2$ it attempts to make a unit
step to the right, and with probability $q = 1 -p < 1/2$ it attempts to make a
unit step to the left, always obeying the exclusion rule that each site can be occupied by at most one particle. Labb\'e and Lacoin \cite{LL}proved that ASEP exhibits cutoff at $(\sqrt{k} + \sqrt{n-k})^2/(p-q)$ with window $n^{1/3}/(p-q)$. Bufetov and Nejjar \cite{BuN} proved that for $\frac{k}{n} \rightarrow \alpha$ the limit profile is given by 
\[\Phi(c)= 1-F_{\textup{GUE}}(c f(\alpha)),\]
where $F_{\textup{GUE}}$ is the GUE Tracy-Widom distribution and $f(\alpha)=\frac{(\alpha(1-\alpha))^{1/6}}{(\sqrt{\alpha} + \sqrt{1-\alpha})^{4/3}}$. The $\ell_2$ analysis around the cutoff time is still open, so we don't know if \eqref{cond} is satisfied, but the limit profile in this case is indeed continuous. 

It is worthwhile to mention that He and Schmid \cite{HS} studied the limit profile for ASEP with one open boundary, and while \eqref{cond} is not known to be satisfied and the limit profile formulas they found are all continuous.

\section{Acknowledgements}
The author would like to thank Dominik Schmid and Lucas Teyssier for their very useful comments.

 \bibliographystyle{plain}
\bibliography{continuity}

\end{document}